\documentclass[a4paper,10pt]{article}
\usepackage{fullpage}
\usepackage[latin1]{inputenc}
\usepackage[T1]{fontenc}
\usepackage{setspace}
\usepackage[normalem]{ulem}
\usepackage{verbatim}
\usepackage{graphicx}
\usepackage{amssymb}
\usepackage{amsmath}
\usepackage{booktabs}
\usepackage{algorithmic}
\usepackage{algorithm}
\usepackage{tikz}
\usetikzlibrary{positioning}
\usetikzlibrary{decorations.pathreplacing}
\usepackage{amsthm}
\usepackage{subfigure}
\usepackage{authblk}
\newtheorem{lem}{Lemma}
\newtheorem{cor}{Corollary}
\newtheorem{theor}{Theorem}
\DeclareMathOperator{\im}{im}
 \DeclareMathOperator{\rank}{rk}

\def\restriction#1#2{\mathchoice
              {\setbox1\hbox{${\displaystyle #1}_{\scriptstyle #2}$}
              \restrictionaux{#1}{#2}}
              {\setbox1\hbox{${\textstyle #1}_{\scriptstyle #2}$}
              \restrictionaux{#1}{#2}}
              {\setbox1\hbox{${\scriptstyle #1}_{\scriptscriptstyle #2}$}
              \restrictionaux{#1}{#2}}
              {\setbox1\hbox{${\scriptscriptstyle #1}_{\scriptscriptstyle #2}$}
              \restrictionaux{#1}{#2}}}
\def\restrictionaux#1#2{{#1\,\smash{\vrule height .8\ht1 depth .85\dp1}}_{\,#2}} 
\usepackage{etoolbox}
\newcommand{\T}{m}
\newcommand{\K}{n}
\newcommand{\n}[1]{%
  \ifstrequal{#1}{1}{_{4}}{}
	\ifstrequal{#1}{2}{_{3}}{}
	\ifstrequal{#1}{3}{_{2}}{}
  \ifstrequal{#1}{4}{_{1}}{}
	\ifstrequal{#1}{i}{_i}{}
}
\newcommand{\p}[1]{P\n{#1}}
\newcommand{\R}[1]{p\n{#1}}
\newcommand{\lgen}[0]{\langle}
\newcommand{\rgen}[0]{\rangle}

\DeclareMathOperator{\olominus}{\overline{\ominus}}
\DeclareMathOperator{\olcap}{\overline{\cap}}
\DeclareMathOperator{\olker}{\overline{\ker}}
\DeclareMathOperator{\olim}{\overline{\im}}

\newcommand{\ttt}{$2\!\times\!2$}
\newcommand{\mypar}[1]{{\bf #1.}}

\title{Generalizing Block LU Factorization:\\A Lower-Upper-Lower Block Triangular Decomposition with Minimal Off-Diagonal Ranks}
\author{Fran\c cois Serre}
\author{Markus P\"uschel}
\affil{Department of Computer Science, ETH Zurich}
\date{}

\begin{document}
\maketitle
\begin{abstract}
We propose a novel factorization of a non-singular matrix $P$, viewed as a $2\times 2$-blocked matrix. The factorization decomposes $P$ into a product of three matrices that are lower block-unitriangular, upper block-triangular, and lower block-unitriangular, respectively. Our goal is to make this factorization ``as block-diagonal as possible'' by minimizing the ranks of the off-diagonal blocks. We give lower bounds on these ranks and show that they are sharp by providing an algorithm that computes an optimal solution. The proposed decomposition can be viewed as a generalization of the well-known Block LU factorization using the Schur complement.
Finally, we briefly explain one application of this factorization: the design of optimal circuits for a certain class of streaming permutations.
\end{abstract}

\section{Introduction}
Given is a non-singular matrix $\p0\in GL_{\T+\K}(\mathbb{K})$ over a field $\mathbb{K}$. We partition $\p0$ as
$$
\p0=\begin{pmatrix}
\p4 & \p3\\
\p2 & \p1\\
\end{pmatrix},
\quad\text{such that $\p4$ is $\T\times \T$}.
$$
We denote the ranks of the submatrices with $\R{i}=\rank \p{i}$, $i=1,2,3,4$. Matrices are denoted with capital letters and vector spaces with calligraphic letters.

If $\p4$ is non-singular, then a block Gaussian elimination uniquely decomposes $\p0$ into the form:
\begin{equation}\label{deco1}\p0=\begin{pmatrix}
I_\T & \\
L & I_\K\\
\end{pmatrix}
 \begin{pmatrix}
C\n4 & C\n3\\
 & C\n1\\
\end{pmatrix},\end{equation}
where $I_\T$ denotes the $\T\times \T$ identity matrix. The rank of $L = \p2\p4^{-1}$ is equal to $\R2$, and $C\n1$ is the Schur complement of $\p4$. Conversely, if such a decomposition exists for $\p0$, then $\p4$ is non-singular. This block LU decomposition has several applications including computing the inverse of $\p0$~\cite{Chow:97}, solving linear systems~\cite{Demmel:92}, and in the theory of displacement structure \cite{Pan:01}.  The Schur complement is also used in statistics, probability and numerical analysis~\cite{Zhang:06,Cottle:74}.

Analogously, the following decomposition exists if and only if $\p1$ is non-singular:
\begin{equation}\label{deco2}\p0= \begin{pmatrix}
C\n4 & C\n3\\
 & C\n1\\
\end{pmatrix}
\begin{pmatrix}
I_\T & \\
R & I_\K\\
\end{pmatrix}.
\end{equation}
This decomposition is again unique, and the rank of $R$ is $\R2$.

In this article, we release the restrictions on $\p1$ and $\p4$ and propose the following decomposition for a general $\p0\in GL_{\T+\K}(\mathbb{K})$:
\begin{equation}\label{deco3}\p0=\begin{pmatrix}
I_\T & \\
L & I_\K\\
\end{pmatrix} \begin{pmatrix}
C\n4 & C\n3\\
 & C\n1\\
\end{pmatrix}
\begin{pmatrix}
I_\T & \\
R & I_\K\\
\end{pmatrix},
\end{equation}
where in addition we want the three factors to be ``as block-diagonal as possible,'' i.e., that $\rank L + \rank C\n3 + \rank R$ is minimal.

\subsection{Lower bounds}
The following theorem provides bounds on the ranks of such a decomposition:
\begin{theor}
\label{bounds}
If a decomposition \eqref{deco3} exists for $\p0\in GL_{\T+\K}(\mathbb{K})$, it satisfies
\begin{align}
\label{bound0}&\rank C\n3=\R3,\\
\label{bound1}&\rank L\geq \K-\R1,\\
\label{bound2}&\rank R\geq \T-\R4,\\
\label{bound3}&\rank R+\rank L \geq \R2.
\end{align}
In particular, the rank of $C\n3$ is fixed and we have:
\begin{equation}
\label{boundS}
\rank R+\rank L \geq \max(\R2,\K+\T-\R1-\R4).
\end{equation}
\end{theor}
We will prove this theorem in Section~\ref{proofTh1}.  Next, we assert that these bounds are sharp. 

\subsection{Optimal solution} The following theorem shows that the inequality \eqref{boundS} is sharp:
\begin{theor}
\label{master}
If $\p0\in GL_{\T+\K}(\mathbb{K})$, then there exists a decomposition \eqref{deco3} that satisfies
$$
\begin{array}{cc}
\rank R+\rank L = \max(\R2,\K+\T-\R1-\R4)\text{ and}\\
\rank L = \K-\R1.
\end{array}
$$
Additionally, such a decomposition can be computed with $O((\T+\K)^3)$ arithmetic operations.
\end{theor}
We prove this theorem in Section~\ref{p2lower} when $\R2\leq \T+\K-\R1-\R4$, and in Section~\ref{p2bigger} for the case $\R2> \T+\K-\R1-\R4$. In both cases, the proof is constructive and we provide a corresponding algorithm (Algorithms~\ref{alg:L1} and \ref{alg:L2}). Theorem~\ref{master} and the corresponding algorithms are the main contributions of this article.

\mypar{Two cases}
As illustrated in Figure~\ref{boundGraph}, two different cases appear from inequality~\eqref{boundS}. If $\R2\leq \T+\K-\R1-\R4$, bound \eqref{bound3} is not restrictive, and the optimal pair $(\rank L, \rank R)$ is unique, and equals $(\K-\R1,\T-\R4)$. In the other case, where $\R2> \T+\K-\R1-\R4$, bound \eqref{bound3} becomes restrictive, and several optimal pairs $(\rank L, \rank R)$ exist.

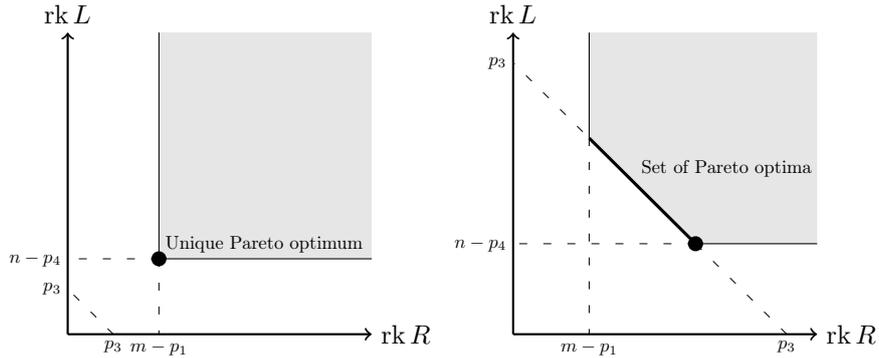
\begin{figure}[!ht]
\centering
\begin{tikzpicture}[scale=0.4]
\draw[->,thick](0,0) -- (10,0);
\draw[->,thick](0,0) -- (0,10);
\draw (10,0) node [right] {$\rank R$};
\draw (0,10) node [above] {$\rank L$};
\path [fill=gray!20] (3,2.5)--(3,10)--(10,10)--(10,2.5)--(3,2.5);
\draw (3,2.5) -- (10,2.5);
\draw (3,2.5) -- (3,10);
\draw[loosely dashed,very thin] (3,2.5)--(3,0);
\draw[loosely dashed,very thin] (3,2.5)--(0,2.5);
\draw (3,0) node [below,scale=0.7] {$\T-\R4$};
\draw (0,2.5) node [left,scale=0.7]  {$\K-\R1$};
\draw[loosely dashed,very thin] (1.5,0)--(0,1.5);
\draw (1.5,0) node [below,scale=0.7] {$\R2$};
\draw (0,1.5) node [left,scale=0.7]  {$\R2$};
\node[circle,inner sep=2pt,fill] at (3,2.5) {};
\draw (3,2.5) node [above right,scale=0.7] {Unique Pareto optimum};
\end{tikzpicture}
\begin{tikzpicture}[scale=0.4]
\draw[->,thick](0,0) -- (10,0);
\draw[->,thick](0,0) -- (0,10);
\draw (10,0) node [right] {$\rank R$};
\draw (0,10) node [above] {$\rank L$};
\path [fill=gray!20] (2.5,6.5)--(6,3)--(10,3)--(10,10)--(2.5,10)--(2.5,6.5);
\draw (6,3) -- (10,3);
\draw (2.5,6.5) -- (2.5,10);
\draw[loosely dashed,very thin] (2.5,6.5)--(2.5,0);
\draw[loosely dashed,very thin] (6,3)--(0,3);
\draw (2.5,0) node [below,scale=0.7] {$\T-\R4$};
\draw (0,3) node [left,scale=0.7]  {$\K-\R1$};
\draw[loosely dashed,very thin] (9,0)--(0,9);
\draw (9,0) node [below,scale=0.7] {$\R2$};
\draw (0,9) node [left,scale=0.7]  {$\R2$};
\draw[very thick] (2.5,6.5)--(6,3);
\draw (4,5) node [above right,scale=0.7] {Set of Pareto optima};
\node[circle,inner sep=2pt,fill] at (6,3) {};
\end{tikzpicture}
\caption{Possible ranks for $L$ and $R$. On the left graph, $\R2< \T+\K-\R1-\R4$. On the right graph, $\R2> \T+\K-\R1-\R4$. The dot shows the decomposition provided by Theorem~\ref{master}.}
\label{boundGraph}
\end{figure}

\mypar{Example}
As a simple example we consider the special case 
$$
\p0=\begin{pmatrix}
&\p3\\
\p2&
\end{pmatrix},\quad\text{with }\K=\T.
$$
In this case $\p2,\ \p3$ are non-singular and neither \eqref{deco1} nor \eqref{deco2} exists.
Theorem~\ref{bounds} gives a lower bound of $\rank R + \rank L \geq 2\K$, which implies that
both $R$ and $L$ have full rank. Straightforward computation shows that for any non-singular $L$,
$$
\p0=\begin{pmatrix}
 I_\K &  \\
 L&  I_\K  
\end{pmatrix}\begin{pmatrix}
L^{-1}\p2 & \p3 \\
 &  -L\p3  
\end{pmatrix}\begin{pmatrix}
 I_\K &  \\
 -(L\p3)^{-1}\p2&  I_\K  
\end{pmatrix}
$$
is an optimal solution. This also shows that the optimal decomposition \eqref{deco3} is in general not unique.

\subsection{Flexibility} 

The following theorem adds flexibility to Theorem~\ref{master} and shows that a decomposition exists for any Pareto-optimal pair of non-diagonal ranks that satisfies the bounds of Theorem~\ref{bounds}:
\begin{theor}
\label{master2}
If $\p0\in GL_{\T+\K}(\mathbb{K})$ and $(l,r)\in \mathbb{N}^2$ satisfies
$l\geq \K-\R1,\ r\geq \T-\R4$, and
$$
r+l= \max(\R2,\K+\T-\R1-\R4),
$$
then $\p0$ has a decomposition \eqref{deco3} with $\rank L=l$ and $\rank R=r$.
\end{theor}
In the case where $\R2\leq \T+\K-\R1-\R4$, the decomposition produced by Theorem~\ref{master} has already the unique optimal pair $(\rank L, \rank R) = (\K-\R1,\T-\R4)$. In the other case, we will provide a method in Section~\ref{rkexc} to trade between the rank of $R$ and the rank of $L$, until bound~\eqref{bound2} is reached. By iterating this method over the decomposition obtained in Theorem~\ref{master}, decompositions with various rank tradeoffs can be built.

Therefore, it is possible to build decomposition \eqref{deco3} for any pair $(\rank L, \rank R)$ that is a Pareto optimum of the given set of bounds. As a consequence, if $f: \mathbb{N}^2\rightarrow \mathbb{R}$ is weakly increasing in both of its arguments, it is possible to find a decomposition that minimizes $f(\rank L,\rank R)$. Examples include $\min(\rank L, \rank R)$, $\max(\rank L, \rank R)$, $\rank L + \rank R$, $\rank L \cdot \rank R$ or $\sqrt{\rank^2 L + \rank^2 R}$.

\mypar{Generalization of block LU factorization}
In the case where $\p1$ is non-singular, Theorem~\ref{master} provides a decomposition that satisfies $\rank L=0$. In other words, it reduces to the decomposition \eqref{deco2}. Using Theorem~\ref{master2}, we can obtain a similar result in the case where $\p4$ is non-singular. Since in this case $\T-\R4=0$, it is possible to choose $r=0$, and thus obtain the decomposition~\eqref{deco1}.

\subsection{Equivalent formulations}

\begin{lem}
The following decomposition is equivalent to decomposition~\eqref{deco3}, with analogous constraints for the non-diagonal ranks:
\begin{equation}
\label{deco4}
\p0=\begin{pmatrix}
I_\T & \\
L\n2 & L\n1\\
\end{pmatrix} \begin{pmatrix}
C\n4 & C\n3\\
 & I_\K\\
\end{pmatrix}
\begin{pmatrix}
I_\T & \\
R\n2 & R\n1\\
\end{pmatrix}
\end{equation}
In this case, the minimization of the non-diagonal ranks is exactly the same problem as in \eqref{deco3}. However, an additional degree of freedom appears: any non-singular $\K\times \K$ matrix can be chosen for either $L\n1$ or $R\n1$.

It is also possible to decompose $\p0$ into two matrices, one with a non-singular leading principal submatrix $L\n4$ and the other one with a non-singular lower principal submatrix $R\n1$: 
\begin{equation}
\label{deco5}
\p0=\begin{pmatrix}
L\n4 & L\n3\\
L\n2 & L\n1\\
\end{pmatrix} 
\begin{pmatrix}
R\n4 & R\n3\\
R\n2 & R\n1\\
\end{pmatrix}
\end{equation}
Once again, the minimization of $\rank L\n2 + \rank R\n2$ is the same problem as in \eqref{deco3}. The two other non-diagonal blocks satisfy $\rank L\n3+\rank R\n3\geq\R3$.
\end{lem}
\begin{proof}
The lower non-diagonal ranks are invariant through the following steps:

\mypar{$\eqref{deco3}\Rightarrow \eqref{deco4}$}
If $\p0$ has a decomposition \eqref{deco3}, a straightforward computation shows that:
$$
\p0=\begin{pmatrix}
I_\T & \\
L & C\n1\\
\end{pmatrix} \begin{pmatrix}
C\n4 & C\n3\\
 & I_\K\\
\end{pmatrix}
\begin{pmatrix}
I_\T & \\
R & I_\K\\
\end{pmatrix},
$$
which has the form of decomposition \eqref{deco4}.

\mypar{$\eqref{deco4}\Rightarrow \eqref{deco5}$}
If $\p0$ has a decomposition \eqref{deco4}, the multiplication of the two left factors leads to formulation \eqref{deco5}. In fact, $L\n4=C\n4$ and $R\n1$ are both non-singular.

\mypar{$\eqref{deco5}\Rightarrow \eqref{deco3}$}
If $\p0$ has a decomposition \eqref{deco5}, then using \eqref{deco1} on the left factor, and \eqref{deco2} on the right factor, and multiplying the two central matrices leads to formulation \eqref{deco3}.
\end{proof}

\subsection{Related work}
\mypar{Schur complement}
Several efforts have been made to adapt the definition of Schur complement in the case of general $\p4$ and $\p1$. For instance, it is possible to define an indexed Schur complement of another non-singular principal submatrix~\cite{Zhang:06}, or use pseudo-inverses~\cite{Carlson:74} for matrix inversion algorithms.

\mypar{Alternative block decompositions}
A common way to handle the case where $\p4$ is singular is to use a permutation matrix $B$ that reorders the columns of $\p0$ such that the new principal upper submatrix is non-singular~\cite{Zhang:06}. Decomposition \eqref{deco1} then becomes:
$$
\p0=\p0BB^T=
\begin{pmatrix}
I_\T & \\
L & I_\K\\
\end{pmatrix}
 \begin{pmatrix}
C\n4 & C\n3\\
 & C\n1\\
\end{pmatrix}
B^T
$$
However, $B$ needs to swap columns with index $\geq m$; thus $B^T$ does not have the required form considered in our work.

One can modify the above idea to choose $B$ such that $B^{-1}$ has the shape required by decomposition \eqref{deco3}:
\begin{eqnarray*}
P & = &
\begin{pmatrix}P\n4 & P\n3\\P\n2 & P\n1\end{pmatrix}
\begin{pmatrix} I_\T \\ -R & I_\K \end{pmatrix} 
\begin{pmatrix} I_\T \\ R & I_\K \end{pmatrix} \\
& = & 
\begin{pmatrix}P\n4 - RP\n3 & P\n3\\P\n2 - RP\n1 & P\n1\end{pmatrix}
\begin{pmatrix} I_\T \\ R & I_\K \end{pmatrix}
\end{eqnarray*}
Then the problem is to design $R$ such that $P\n4 - RP\n3$ is non-singular and $\rank (P\n2 - RP\n1) + \rank R$ is minimal.
This basic idea is used in \cite{Pueschel:09}, where, however, only $\rank R$ is minimized, which, in general, does not produce optimal solutions for the problem considered here.


Finally, our decomposition also shares patterns with a block Cholesky decomposition, or the Block LDL decomposition, in the sense that they involve block uni-triangular matrices. However, the requirements on $\p0$ and the expectations on the decomposition are different.


\section{Application: Optimal Circuits for Streaming Permutations}\label{sec:SLP}

The original motivation for considering our decomposition \eqref{deco3} was an important application in the domain of hardware design. Many algorithms in signal processing and communication consist of alternating computation and reordering (permutation) stages. Typical examples include fast Fourier transforms \cite{Tolimieri:97}; one example (a so-called Pease FFT) is shown in Fig.~\ref{fig:fft} for 16 data points. When mapped to hardware, permutations could become simple wires. However, usually, the data is not available in one cycle, but streamed in chunks over several cycles. Implementing such a ``streaming'' permutation in this scenario on an application-specific integrated circuit (ASIC) or on a field programmable gate array (FPGA) becomes complex, as it now requires both logic and memory \cite{Parhi:92,Pueschel:09}. It turns out that for an important subclass of permutations called ``linear,'' the design of an optimal circuit (i.e., one with minimal logic) is equivalent to solving \eqref{deco3} in the field $\mathbb{F}_2$. 

\begin{figure}\centering
\includegraphics[scale=.45]{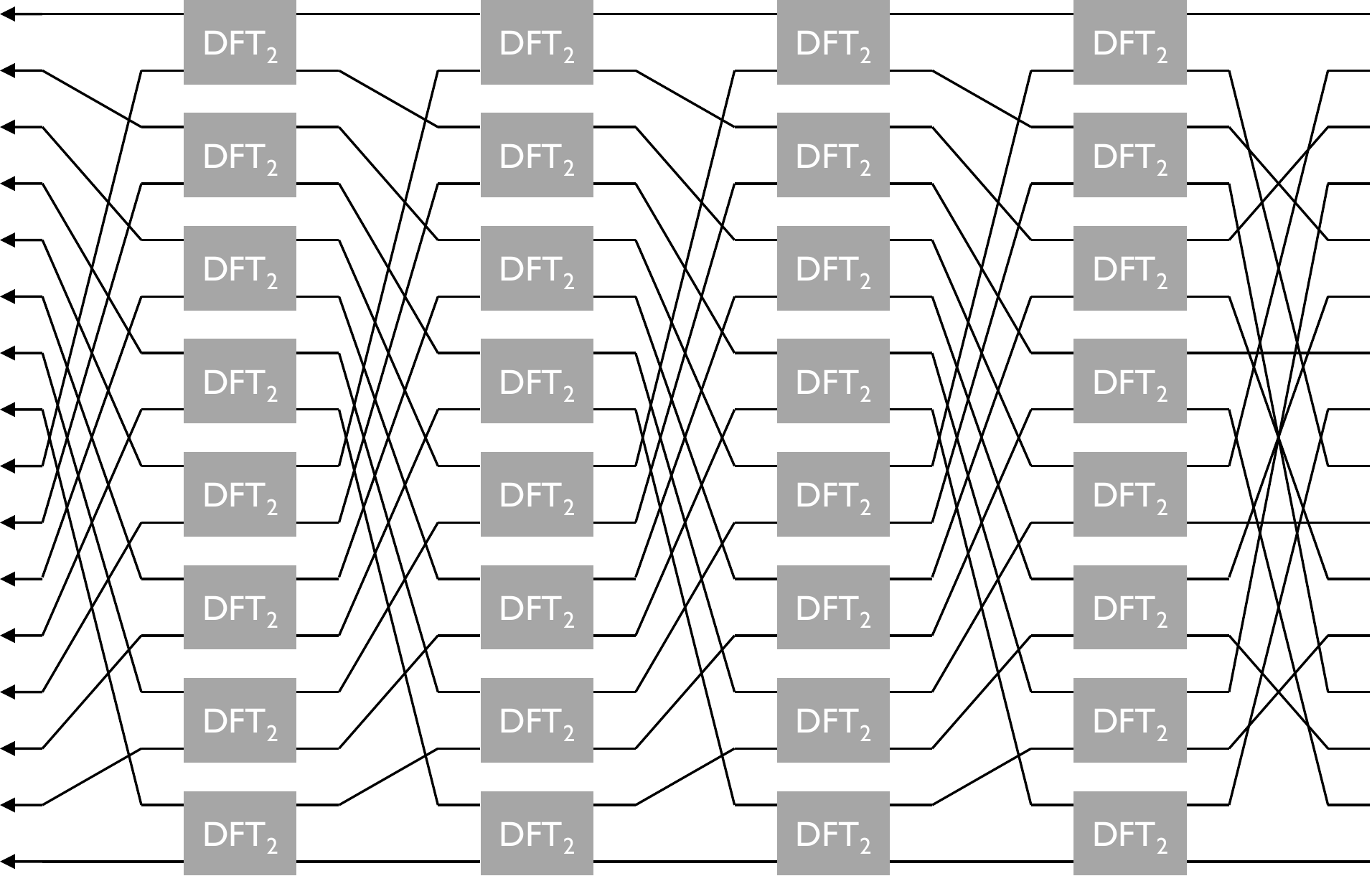}
\caption{Dataflow of a Pease FFT on $16$ points from right (input) to left (output). Computing stages, consisting of DFTs on two points and pointwise multiplications (not shown), alternate with permutations.}\label{fig:fft}
\end{figure}

Next, we provide a few more details on this application starting with the necessary background information. However, we only provide a sketch; a more complete treatment can be found in \cite{Pueschel:09}.


\subsection{Linear permutations}

For $0\leq i < 2^{\T+\K}$, we denote with $i_b$ the associated (bit) vector in $\mathbb{F}_2^{\T+\K}$ that contains the radix-$2$ digits of $i$, with the most significant digit on top. For instance, for $m+n=4,$ we have $$12_b= \begin{pmatrix}1\\1\\0\\0\end{pmatrix}\text{ and } 7_b=\begin{pmatrix}0\\1\\1\\1\end{pmatrix}.$$ Any invertible $(\T+\K)\times (\T+\K)$ matrix $P$ over $\mathbb{F}_2$ induces a permutation $\pi(P)$ on $\{0,\dots,2^{\T+\K}-1\}$ that maps $i$ to $j$, where $j_b=P\cdot i_b$. 

For example, if we define $$V_{\T+\K}=\begin{pmatrix}1&&\\\vdots&\ddots&\\1&&1\end{pmatrix},$$ then $\pi(V_3)$ is the mapping $0\mapsto 0$, $1\mapsto 1$,$2\mapsto 2$,$3\mapsto 3$,$4\mapsto 7\mapsto 4$, $5\mapsto 6 \mapsto 5$. More generally, $\pi(V_{\T+\K})$ is the permutation that leaves the first half of the elements unchanged, and that reverts the second half.

The mapping $\pi: GL_{\T+\K}(\mathbb{F}_2) \rightarrow S_{2^{\T+\K}}$ is a group-homomorphism, and its range is called the group of linear permutations \cite{Pease:77,Lenfant:85}. This group contains many of the permutations used in signal processing and communication algorithms, including stride permutations, bit-reversal, Hadamard reordering, and the Grey code reordering.

\subsection{Streamed linear permutations (SLP)}

We want to implement a circuit that performs a linear permutation on $2^{\T+\K}$ points. If we assume that this circuit has $2^\K$ input (and output) ports, this means that the dataset has to be split into $2^\T$ parts that are fed (streamed) as input over $2^\T$ cycles. Similarly, the permuted output will be produced streamed over $2^\T$ cycles. As an example consider Fig.~\ref{fig:streamed} in which $2^\K = 2$ and $2^\T=4$.

\begin{figure}\centering
   \includegraphics[scale=.8]{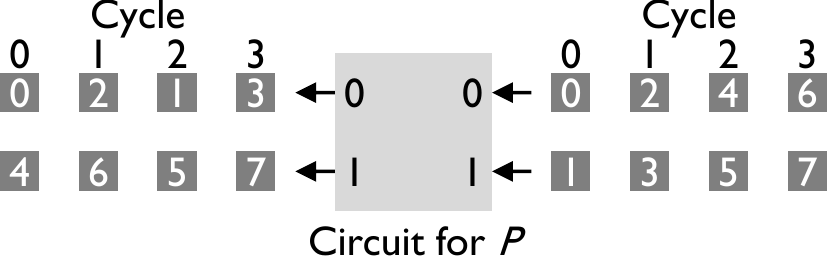}
\caption{Sketch of an implementation of the bit reversal permutation on $2^3$ elements streamed on two ports. The dataset enters within $4$ cycles (right), and is retrieved within $4$ cycles (left).}\label{fig:streamed}
\end{figure}

With this convention, the element with the index $i=c\cdot 2^\T + p$ arrives as input in the $c^{th}$ cycle at the $p^{th}$ port. Particularly, the upper $\T$ bits of $i_b$ are the bit representation $c_b$ of the input cycle $c$, while the lower $\K$ bits are $p_b$. For example, if $\T=3$ and $\K=2$, then the element indexed with $$14_b=\begin{pmatrix}0\\1\\1\\\hline 1\\0\end{pmatrix}=\begin{pmatrix}3_b\\\hline 2_b\end{pmatrix}$$ will arrive during the third cycle on the second port.

Thus, it is natural to block the desired linear permutation
$$
P=\begin{pmatrix}\p4 & \p3 \\ \p2 & \p1\end{pmatrix}
$$ 
such that $\p4$ is $\T\times \T$. This implies that the element that arrives on port $p$ during the $c^{th}$ cycle has to be routed to the output port $\p2 c_b + \p1 p_b$ at output cycle $\p4 c_b + \p3 p_b$.

$\p2$ has a particular role here, as it represents ``how the routing between the different ports must vary during time'', and directly influences the complexity of the implementation. For example, if $\p2 = 0$, then the output is always $\p1 p_b$ without variation during time.


\subsection{Implementing SLPs on hardware}

Two special cases of SLPs can be directly translated to a hardware implementation. The first kind are the permutations that only permute across time, i.e., that do not change the port number of the elements. Thus, they satisfy $\p2 c_b + \p1 p_b = p_b$ for all $c$ and $p$, and therefore have the form 
$$
P=\begin{pmatrix}\p4 & \p3 \\ & I_\K\end{pmatrix}.
$$ 
These SLPs can be implemented using an array of $2^\K$ blocks of RAMs as shown in Fig.~\ref{fig:RAM}.

\begin{figure}\centering
\subfigure[RAM bank array]{
   \includegraphics[scale=.6]{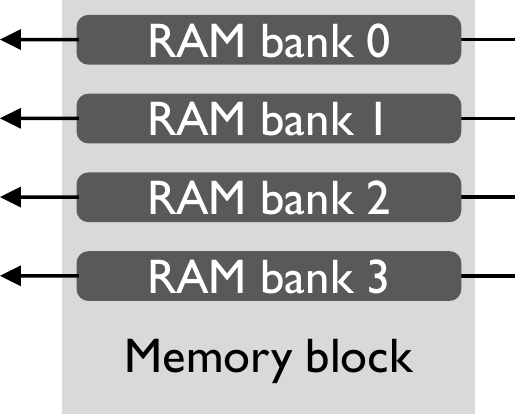}\label{fig:RAM}}
	\quad\quad
\subfigure[Switching network]{
   \includegraphics[scale=.6]{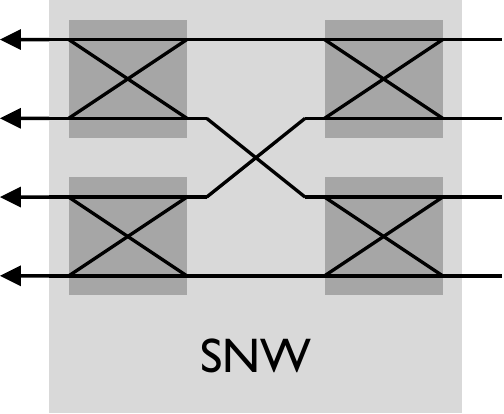}\label{fig:SWN}}
\caption{Two hardware blocks to implement particular types of SLPs. On the left, an array of $4$ RAM banks. Write and read addresses are used to delay differently elements. This block can be used to implement SLPs that do not permute across ports. On the right is an example of a switching network consisting of $4$ \ttt-switches. Switching networks can be used to implement SLPs that do not permute across time.}
\end{figure}

Conversely, SLPs that only permute across the ports within each cycle have the form 
$$
P=\begin{pmatrix}I_\T &  \\ \p2 & \p1\end{pmatrix}.
$$ 
They do not require memory and can be implemented using a network of $\R2\cdot 2^{\K-1}$ \ttt-switches as was shown in \cite{Pueschel:09}. This result, combined with decomposition~\eqref{deco4} and Theorem~\ref{master} yields an immediate corollary:

\begin{theor}
For a given $P$, the associated SLP can be implemented using a RAM bank array (Fig.~\ref{fig:RAM}) framed by two switching networks (Fig.~\ref{fig:SWN}) with a total of $\max(\R2,\K+\T-\R1-\R4)\cdot 2^{\K-1}$ \ttt-switches.
\end{theor}
Our Algorithms~\ref{alg:L1} and \ref{alg:L2}, introduced later, provide an efficient method to compute this architecture. 


%
%

\section{Preliminaries}\label{preliminaries}

In this section, we will prove some basic lemmas that we will use throughout this article.

\subsection{Properties of the blocks of an invertible matrix}

In this subsection, we derive some direct consequences of the invertibility of $\p0$ on the range and the nullspace of its submatrices.
\begin{lem}
The following properties are immediate from the structure of $\p0$:
\begin{equation}
\label{P1P3}
\ker \p1 \cap \ker \p3 = \{0\}
\end{equation}
\begin{equation}
\label{P2P4}
\ker \p2 \cap \ker \p4 = \{0\}
\end{equation}
\begin{equation}
\label{P1P2}
\im \p1+\im \p2 = \mathbb{K}^{\K}
\end{equation}
\begin{equation}
\label{P3P4}
\im \p3+\im \p4 = \mathbb{K}^{\T}
\end{equation}
\begin{equation}
\label{P1kerP3P2kerP4}
\p2(\ker \p4)\cap \p1(\ker \p3)=\{0\}
\end{equation}
\begin{equation}
\label{P3kerP1P4kerP2}
\p4(\ker \p2)\cap \p3(\ker \p1)=\{0\}
\end{equation}
\end{lem}

\begin{proof}
We prove here equation~\eqref{P1kerP3P2kerP4}. If $x\in \ker \p4$ and $y\in \ker \p3$ satisfy $\p2x=\p1y$, we have:$$\begin{pmatrix}
\p4 & \p3\\
\p2 & \p1\\
\end{pmatrix}\cdot \begin{pmatrix}
x\\
-y\\
\end{pmatrix}=\begin{pmatrix}
\p4x-\p3y\\
\p2x-\p1y \\
\end{pmatrix}=\begin{pmatrix}
0\\
0\\
\end{pmatrix}$$
Since $\p0$ is non-singular, $x=y=0$, as desired.
\end{proof}

These equalities yield the dimensions of the following subspaces:

\begin{cor}
\begin{equation}
\label{P3kerP1}
\dim \p3(\ker \p1)=\dim \ker \p1 = \K-\R1
\end{equation}
\begin{equation}
\dim \p4(\ker \p2)=\dim \ker \p2 = \T-\R2
\end{equation}
\begin{equation}
\dim \p1(\ker \p3)=\dim \ker \p3 = \K-\R3
\end{equation}
\begin{equation}
\label{P2kerP4}
\dim \p2(\ker \p4)=\dim \ker \p4 = \T-\R4
\end{equation}
\begin{equation}
\label{P1interP2}
\dim \im \p1\cap \im \p2=\R1+\R2-\K
\end{equation}
\begin{equation}
\dim \im \p3\cap \im \p4=\R3+\R4-\T
\end{equation}
\end{cor}

\begin{proof}
We prove equation~\eqref{P3kerP1}. Since, by \eqref{P1P3}, $\ker \p1 \cap \ker \p3 = \{0\}$, the dimension of the image of $\ker\p1$ under $\p3$ has the same dimension as $\ker\p1$, which is $n-p_4$.

Equation~\eqref{P1interP2} is a consequence of equation~\eqref{P1P2} and of the rank-nullity theorem.
\end{proof}

\subsection{Algorithms on linear subspaces in matrix form}

The algorithms we present in this article heavily rely on operations on subspaces of $\mathbb{K}^m$. To make the representation of these algorithms more practical for implementation, we introduce a matrix representation for subspaces and formulate the subspace operations needed in this paper on this representation.


We represent a linear subspace as an associated matrix\footnote{We allow the existence of matrices with $0$ column to represent the trivial subspace of $\mathbb{K}^m$.} whose columns form a basis of this subspace. In other words, if $\mathcal{A}$ is a subspace of $\mathbb{K}^m$ of dimension $n$, then we represent it using a $m\times n$ matrix $A$ such that $\im A=\mathcal{A}$. In this case, and only in this case, we will use the notation $\lgen A \rgen = \im A$ to emphasize that the columns of $A$ form a linear independent set.

With this notation we can formulate subspace computations as computations on their associated matrices as explained in the following. To formally emphasize this correspondence, these operations on matrices will carry the same symbol as the subspace computation (e.g., $\cap$) they represent augmented with an overline (e.g., $\olcap$). The operations are collected in Table~\ref{MatOps}. All algorithms in this paper are written as sequences of these matrix operations. Because of this, we implemented the algorithms by first designing an object oriented infrastructure that provides these operations. Then we could directly map the algorithms, as they are formulated, to code.

\begin{table}\small\center
\begin{tabular}{@{}lll@{}}
\toprule
Operation related to subspaces & Associated matrix operation & Correspondence\\
\midrule

Kernel of a matrix $M$ &
$\olker M$ &
$\lgen\olker M\rgen=\ker M$\\

Direct sum of subspaces $\mathcal{A},\mathcal{B}$ &
$\begin{pmatrix}A & B\end{pmatrix}$ &
$\begin{pmatrix}A & B\end{pmatrix} = \mathcal{A}\oplus\mathcal{B}$\\

Intersection of subspaces $\mathcal{A},\mathcal{B}$&
$A\olcap B$&
$\lgen A\olcap B\rgen = \mathcal{A}\cap\mathcal{B}$\\

Complement of a subspace $\mathcal{B}$ in $\mathcal{A}$&
$A\olominus B$ &
$\lgen A\olominus B\rgen\oplus  \mathcal{B}=\mathcal{A}$\\
\bottomrule
\end{tabular}
\caption{\label{MatOps} We summarize matrix operators to perform basic operations on subspaces. $M$ is a general matrix, while $A$ and $B$ are matrices with $m$ rows that represent the subspaces $\mathcal{A}$ and $\mathcal{B}$; i.e. $\mathcal{A} = \lgen A\rgen$ and $\mathcal{B} = \lgen B\rgen$. Inspection of these routines shows that they all can be implemented with $O(m^3)$ runtime.}
\end{table}

\mypar{Direct sum of two subspaces}
If $\lgen A\rgen, \lgen B\rgen\leq\mathbb{K}^m$ are two subspaces, then the direct sum can be computed by concatenating the two matrices: $\lgen A \rgen \oplus\lgen B\rgen=\lgen\begin{pmatrix}A & B\end{pmatrix}\rgen$. 

\mypar {Null space of a matrix}
Gaussian elimination can be used to compute the null space of a given $m\times n$ matrix $M$. Indeed, if the reduced column echelon form of the matrix $\begin{pmatrix}M\\I_n\end{pmatrix}$ is blocked into the form $\begin{pmatrix}M\n4 & \\ M\n2 & M\n1\end{pmatrix},$ where $M\n4$ has $m$ rows and no zero column, then $\ker M = \lgen M\n1\rgen$. We denote this computation with $\olker M = M_4$.

\mypar {Intersection of two subspaces}
For two subspaces $\im A, \im B\leq\mathbb{K}^m$, the intersection can be computed using the Zassenhaus algorithm. Namely, if the reduced column echelon form of $\begin{pmatrix}A & B \\ A & \end{pmatrix}$ is blocked into the form $\begin{pmatrix}C\n4 & &\\ C\n2 & C\n1& \end{pmatrix},$ where $C\n4$ and $C\n1$ have $m$ rows and no zero column, then $\im A\cap\im B = \lgen C\n1\rgen$. We denote this computation with $A\olcap B = C_4$.


\mypar {Complement of a subspace within another one}
Let $\im A\leq\im B\leq \mathbb{K}^m$. Then, a complement $\lgen C\rgen$ of $\im A$ in $\im B$, i.e., a space that satisfies $\lgen C\rgen \oplus \im A = \im B$, can be obtained as described in Algorithm~\ref{alg:sc}\footnote{This algorithm can be implemented to run in a cubic arithmetical time by keeping a reduced column echelon form of $\begin{pmatrix}A & C\end{pmatrix}$, which makes it possible to check the condition within the loop in a squared arithmetical time.}. We denote this operation with $C = B\olominus A$.

\begin{algorithm}\small
\caption{Complement of a subspace within another}
\label{alg:sc}
\begin{algorithmic}
	\REQUIRE Two matrices $A$ and $B$ with $m$ rows
	\ENSURE A matrix $A \olominus B$ such that $\lgen A \olominus B\rgen \oplus \im B=\im A$
	\STATE $C \leftarrow$ A $m\times 0$ matrix
	\FOR {\textbf{each} column vector $b$ of $B$}
	\IF {$\rank \begin{pmatrix}A & C & b\end{pmatrix}>\rank \begin{pmatrix}A & C\end{pmatrix}$}
	\STATE $C \leftarrow \begin{pmatrix}C & b \end{pmatrix}$
	\ENDIF
	\ENDFOR
	\RETURN $C$
\end{algorithmic}
\end{algorithm}

\subsection{Double complement}

\begin{lem}
\label{dcomplement}
Let $\mathcal{C}$ be a finite-dimensional vector space and let $\mathcal{A},\mathcal{B}\leq \mathcal{C}$ with $\dim \mathcal{A} \geq \dim \mathcal{B}$. Then, there exists a space $\mathcal{S}\leq \mathcal{C}$ such that:
$$\left\lbrace \begin{array}{l} 
\mathcal{S}\oplus \mathcal{A}=\mathcal{C}\\ 
\mathcal{S}\cap \mathcal{B}=\{0\}
\end{array}\right.
$$
\end{lem}

\begin{proof}
We first consider the case where $\mathcal{C}=\mathcal{A}+\mathcal{B}$. We denote with $\mathcal P$ (resp. $\mathcal Q$) a complement of $\mathcal{A}\cap\mathcal{B}$ in $\mathcal A$ (resp. in $\mathcal B$). 
\begin{center}
\begin{tikzpicture}
    \node (top) at (0,0) {$\mathcal{C}=\mathcal{A}+\mathcal{B}$};
    \node (left) at (-1,-1)  {$\mathcal{A}$};
    \node (right) at (1,-2) {$\mathcal{B}$};
    \node (bottom) at (0,-3) {$\mathcal{A}\cap\mathcal{B}$};
    \node (left) at (-1,-1)  {$\mathcal{A}$};
    \draw [black, thick, shorten <=-2pt, shorten >=-2pt] (top) -- (left);
    \draw [black, thick, shorten <=-2pt, shorten >=-2pt] (top) -- (right);
    \draw [black, thick, shorten <=-2pt, shorten >=-2pt] (bottom) edge node[left] {$\mathcal{P}$} (left);
    \draw [black, thick, shorten <=-2pt, shorten >=-2pt] (bottom) edge node[right] {$\mathcal{Q}$} (right);
\end{tikzpicture}
\end{center}

We show first that $\mathcal{P}\cap\mathcal{Q}=\{0\}$. Let $v\in \mathcal{P}\cap\mathcal{Q}$. As $\mathcal{P}\leq\mathcal{A}$ and $\mathcal{Q}\leq\mathcal{B}$, we have $v\in \mathcal{A}\cap\mathcal{B}$. Therefore, $v\in \mathcal P\cap \mathcal A \cap \mathcal B=\{0\},$ as desired.

We now denote with $b=\{b_1,\dots, b_p\}$ (resp. $b'=\{b'_1,\dots, b'_q\}$) a basis of $\mathcal P$ (resp. $\mathcal Q$), implying $q\leq p$. Considering $w=\{b_1+b'_1,\dots , b_q+b'_q\}$, the following holds:
\item[i)] $w$ is linear independent: if $\{\alpha_1,\dots ,\alpha_q\}\in \mathbb{K}^q$ is such that $\sum{\alpha_i w_i}=0$, then $\sum{\alpha_i b_i}=\sum{\alpha_i b'_i}$. As $\sum{\alpha_i b_i}\in \mathcal{P}$ and $\sum{\alpha_i b'_i}\in \mathcal{Q}$, it comes that $\sum\alpha_i b'_i=\sum\alpha_i b_i=0$. It follows that for all $i$, $\alpha_i=0$, yielding the result.
\item[ii)] $\langle w\rangle \cap \mathcal{A} =\{0\}$: If $v\in \langle w\rangle \cap \mathcal{A}$, then there exists $\{\alpha_1,\dots ,\alpha_q\}\in \mathbb{K}^q$ such that $v=\sum{\alpha_i(b_i+b'_i)}\in \mathcal{A}$. It implies $\sum{\alpha_ib'_i}\in\mathcal{A}$. As the left hand side is in $\mathcal{Q}$, it comes that $\sum{\alpha_ib'_i}=0$. It follows that for all $i$, $\alpha_i=0$, yielding the result.
\item[iii)] $\langle w\rangle \cap \mathcal{B} =\{0\}$: Same proof as above.

Then, as $\dim \langle w\rangle=q=\dim \mathcal{C}-\dim \mathcal{A}$, $\mathcal S = \langle w\rangle$ satisfies the desired conditions.

In the general case, where $\mathcal C > \mathcal A + \mathcal B$, we use the same method, and simply add a complement $\mathcal S'$ of $\mathcal A + \mathcal B$ in $\mathcal C$ to the solution.

\end{proof}
Algorithm~\ref{alg:dc} uses the method used in this proof to compute a basis of $\mathcal{S}$, given $\mathcal{A}$, $\mathcal{B}$ and $\mathcal{C}$. Note that if $\dim \mathcal{A}=\dim \mathcal{B}$, then $\mathcal {S}$ is a complement of both $\mathcal{A}$ and $\mathcal{B}$ in $\mathcal{C}$.

\begin{algorithm}\small
\caption{``Double complement'' algorithm (Lemma~\ref{dcomplement})}
\label{alg:dc}
\begin{algorithmic}
	\REQUIRE $A$, $B$ and $C$, such that $\im A, \im B \leq\im C$ and $\rank A \geq \rank B$
	\ENSURE A matrix $S$ such that $\lgen S\rgen \oplus \im A=\im C$ and $\lgen S\rgen \cap \im B=\{0\}$.
	\STATE $P \leftarrow A \olominus (A\olcap B)$
	\STATE $Q \leftarrow B \olominus (A\olcap B)$
	\STATE $P' \leftarrow P$ truncated such that $P'$ and $Q$ have the same size
	\STATE $S'\leftarrow C \olominus \begin{pmatrix}A& B\end{pmatrix}$
	\RETURN $\begin{pmatrix}S' & P' + Q \end{pmatrix}$
\end{algorithmic}
\end{algorithm}

\section{Proof of Theorem~\ref{bounds}}\label{proofTh1}

We start with an auxiliary result that asserts that a decomposition of the form \eqref{deco3} is characterized by $L$.
\begin{lem}
\label{rewriting}
Decomposition \eqref{deco3} exists if and only if $L$ is chosen such that $\p1-L\p3$ is non-singular. In this case, 
\begin{equation}
\label{rkR}
\rank R = \rank (\p2-L\p4).
\end{equation}
\end{lem}

\begin{proof} We have:
\begin{equation}
\label{decor}
\begin{pmatrix}
I_\T &\\
L & I_\K
\end{pmatrix}^{-1}P=\begin{pmatrix}
\p4 & \p3\\
\p2-L\p4 & \p1-L\p3
\end{pmatrix}\end{equation}
This matrix can be uniquely decomposed as in \eqref{deco2} if and only if $\p1-L\p3$ is non-singular, and we have the desired value for $\rank R$.
\end{proof}

Now we start the actual proof of Theorem~\ref{bounds}. If we assume that decomposition \eqref{deco3} exists for $\p0$, then Lemma~\ref{rewriting} yields
\begin{align}
\label{expdeco}
P=\begin{pmatrix}
I_\T &\\
L & I_\K
\end{pmatrix}\cdot\begin{pmatrix}
\p4-\p3(\p1-L\p3)^{-1}(\p2-L\p4)&\p3\\
&\p1-L\p3
\end{pmatrix}\cdot\notag\\\begin{pmatrix}
I_\T &\\
(\p1-L\p3)^{-1}(\p2-L\p4) & I_\K
\end{pmatrix}
\end{align}

It follows:
\begin{itemize}
\item \eqref{bound0} is obvious from \eqref{expdeco}.
\item $\mathbb{K}^\K=\im (\p1 -L\p3) \leq \im \p1 + \im L$. Thus, $\K\leq \R1+\rank L$, which yields \eqref{bound1}.
\item $\im (\p2-L\p4)=(\p2-L\p4)(\mathbb{K}^{\K})\geq(\p2-L\p4)(\ker \p4) = \p2(\ker \p4)$. \eqref{bound2} now follows from \eqref{P2kerP4} and \eqref{rkR}.
\item \eqref{bound3} is a direct computation: 
$$\begin{array}{rcl}
\R2 & = & \rank(\p2-L\p4+L\p4)\\
&\leq & \rank (\p2-L\p4)+\rank (L\p4)\\
&\leq & \rank R+\rank L.\end{array}$$
\end{itemize}

\section{Proof of Theorem~\ref{master}, case $\R2\leq \T+\K-\R1 - \R4$}

\label{p2lower}
In this section, we provide an algorithm to construct an appropriate decomposition, in the case where $\R2\leq \T+\K-\R1 - \R4$ (Figure~\ref{boundGraph} left). This means that, using Lemma \ref{rewriting}, we have to build a matrix $L$ that satisfies
$$
\begin{cases}
\p1-L\p3 \text{ is non-singular},\\
\rank L = \K-\R1,\\
\rank (\p2-L\p4) = \T-\R4.
\end{cases}
$$

\subsection{Sufficient conditions}
We first derive a set of sufficient conditions that ensure that $L$ satisfies the two following properties: $\p1-L\p3$ non singular (Lemma~\ref{isSol}) and $\rank (\p2-L\p4)=\T-\R4$ (Lemma~\ref{P2SP4min1}).

\begin{lem}
\label{isSol}
If $\rank L = \K-\R1$ and $\im \p1\oplus L\p3(\ker \p1) =\mathbb{K}^\K$, then $\p1-L\p3$ is non-singular.
\end{lem}

\begin{proof}
We denote with $\mathcal{U}$ a complement of $\ker \p1$ in $\mathbb{K}^{\K}$, i.e., $\mathbb{K}^{\K}=\ker \p1\oplus\mathcal{U}$. This implies $\im \p1=\p1(\mathbb{K}^{\K})=\p1(\mathcal{U})$.
Now, let $\im \p1\oplus L\p3(\ker \p1)=\mathbb{K}^\K$. Hence, $\dim L\p3(\ker \p1)=\K-\R1=\rank L$ from which we get $\im L=L\p3(\ker \p1)$. In particular, $L\p3(\mathcal{U})\leq \im L = L\p3(\ker \p1)$. Further, 
$$\begin{array}{rcl}
\im (\p1-L\p3) & = & (\p1-L\p3)(\mathcal{U} \oplus \ker \p1)\\
& = & (\p1-L\p3)(\mathcal{U})+L\p3(\ker \p1)
\end{array}$$
As $L\p3(\mathcal{U})\leq L\p3(\ker \p1)$ and $\im \p1\cap L\p3(\ker \p1)=\{0\}$, we have:
$$\begin{array}{rcl}
\im (\p1-L\p3) & = & \p1(\mathcal{U})+L\p3(\ker \p1)\\
& = & \im \p1 +L\p3(\ker \p1)\\
& = & \mathbb{K}^\K\end{array}$$ as desired.
\end{proof}

\begin{lem}
\label{P2SP4min1}
If, for every vector $v$ of $\im \p4$, $L$ satisfies $Lv \in \p2\p4^{-1}(\{v\})$, then $\rank (\p2-L\p4)=\T-\R4$.
\end{lem}
\begin{proof}
Let $Lv \in \p2\p4^{-1}(\{v\})$ for all $v\in \im \p4$. If $u\in \mathbb{K}^\T$, we have:
$$\begin{array}{rcl}
(\p2-L\p4)u & = & \p2u-L\p4u\\
& \in & \p2u-\p2\p4^{-1}(\{\p4u\})\\
& \in & \p2u-\p2(u+\ker \p4)\\
& \in & \p2(\ker \p4)
\end{array}$$
Therefore, $\im (\p2-L\p4) = \p2(\ker \p4)$. Thus, $\rank \p2-L\p4=\T-\p4$.
\end{proof}

The following lemma summarizes the two previous results:
\begin{lem}
\label{cond1}
Let $\mathcal{Y}$ be a complement of $\im \p1$ and $\mathcal{T}$ a complement of $\p4(\ker \p2)$ in $\im \p4$. If
$$
\begin{cases}
\im L = \mathcal{Y},\\
L\cdot \p3(\ker \p1)=\mathcal{Y},\\
L\cdot v\in \p2\p4^{-1}(\{v\}), \forall v \in \mathcal{T},\\
L\cdot \p4(\ker \p2) = \{0\},\\
\end{cases}
$$
then $L$ is an optimal solution\footnote{The proposed set of sufficient conditions is stronger than necessary; if we replace the last condition with $L\cdot \p4(\ker \p2)\leq \p2(\ker \p4)$, if and only if holds.}.
\end{lem}

\subsection{Building $L$}

We now build a matrix $L$ that satisfies the previous set of sufficient conditions. For all $v$ in a complement $\mathcal{T}$ of $\p4(\ker \p2)$, $L$ has to satisfy $Lv \in \p2\p4^{-1}(\{v\})$. We first show that, given a suitable domain and image, it is possible to build a bijective linear mapping that satisfies this property. 

\begin{lem}
\label{build}
Let $\p4(\ker \p2)\oplus \mathcal{T}=\im \p4$ and $\p2(\ker \p4)\oplus \mathcal{V}=\im \p2$. If $$\mathcal{F}=\p4^{-1}(\mathcal{T})\cap\p2^{-1}(\mathcal{V}),$$ then the mapping $f:\mathcal{T}\rightarrow\mathcal{V}$ such that for all $v\in \mathcal{F}, f(\p4 v)= \p2 v$ is well defined and is an isomorphism.
\end{lem}

\begin{proof}
We prove the lemma by first considering two functions $f_1$ and $f_2$ that are $P_1$ and $P_3$ restricted to $\mathcal{F}$ as shown in the diagram. We show that both are isomorphisms. Then $f=f_2\circ f_1^{-1}$ is the desired function.

\begin{center}
\begin{tikzpicture}
		\node (T) at (0,0) {$\mathcal{T}$};
		\node (V) at (1.5,0) {$\mathcal{V}$};		
		\node (F) at (0.75,-1.25) {$\mathcal{F}$};		
		\draw [->] (T)  edge node[above] {$f$}(V);
				\draw [->] (F)  edge node[left] {$f_1:\ x\mapsto P_1x$}(T);
						\draw [->] (F)  edge node[right] {$f_2:\ x\mapsto P_3x$}(V);
\end{tikzpicture}
\end{center}

We begin with the surjectivity of $f_1$. Let $x\in \mathcal{T}$. As $\mathcal{T}\leq\im\p4$, there exists a vector $v$ such that $\p4 v=x$. The coset $v+\ker \p4$ is obviously a subset of $\p4^{-1}(\mathcal{T})$. Additionally, its image under $\p2$, the coset $\p2(v+\ker \p4)=\p2 v +\p2\ker\p4$ contains a unique representative $\p2 v_f$ in $\mathcal{V}$, since $\im \p2 = \p2(\ker \p4)\oplus \mathcal{V}$. Therefore, $v_f\in\p2^{-1}(\mathcal{V})$, and thus $v_f\in \mathcal{F}$ and $f_1(v_f)=x,$ as desired.

We now prove that $f_1$ is injective. Let $v\in\ker f_1\leq\mathcal{F}$. We have $\p2v\in\mathcal{V}$. Since $v\in\ker\p4$, $\p2v\in\p2\ker\p4$. Since $\p2(\ker \p4)\cap \mathcal{V}=\{0\}$, $\p2v = 0$ and thus $v\in\ker\p2$. Equation~\eqref{P2P4} shows that $v=0$, as desired. Thus, $f_1$ is bijective.

The proof that $f_2:\ \mathcal{F}\to\mathcal{V},\ v \mapsto \p2 v$ is bijective is analogous. It follows that $f=f_2\circ f_1^{-1}$ is the desired isomorphism.
\end{proof}

\label{Final1}

As explained below, we now build a matrix $L$ that matches the conditions listed in Lemma~\ref{cond1}. As they involve two spaces that may not be in a direct sum, $\p3(\ker \p1)$ and a complement of $\p4(\ker \p2)$ in $\im \p4$, some precautions must be taken.

We first construct the image $\mathcal{Y}$ of $L$. It must be a complement of $\im \p1$ and must contain a complement $\mathcal{Y}_1$ of $\p2(\ker \p4)$ in $\im \p2$. As $\R2\leq \T+\K-\R1-\R4$, $\T-\R4\geq \R1 +\R2 - \K$, and we have $\dim(\p2(\ker \p4)) \geq \dim (\im \p1 \cap \im \p2)$. Therefore, we can use the Lemma~\ref{dcomplement} to build a space $\mathcal{Y}_1$ such that:
$$\left \{\begin{array}{l}
\mathcal{Y}_1\oplus \p2(\ker \p4)=\im \p2,\\
\mathcal{Y}_1\cap \im \p1 \cap \im \p2=\{0\}.
\end{array}\right .$$ 
We then complete $\mathcal{Y}_1$ to form a complement $\mathcal{Y}$ of $\im \p1$.

We now decompose $\mathbb{K}^\T$ the following way:
\begin{center}
\begin{tikzpicture}
		\node (Z) at (0,0) {$\mathcal{X}_1$};
		\node [anchor=north west] (op1) at (Z.north east){$\oplus$};
		\node [anchor=north west](X) at (op1.north east){$\mathcal{X}_2$};
		\node [anchor=north west](op2) at (X.north east){$\oplus$};
		\node [anchor=north west](Y) at (op2.north east){$\mathcal{X}_3$};
		\node [anchor=north west](op3) at (Y.north east){$\oplus$};
		\node [anchor=north west](P4kerP2) at (op3.north east) {$\p4(\ker \p2)$};
		\node [anchor=north west](op4) at (P4kerP2.north east){$\oplus$};
		\node [anchor=north west](X4) at (op4.north east) {$\mathcal{X}_4$};
		\node [anchor=west](op5) at (X4.east) {$=$};
		\node [anchor=west]at (op5.east) {$\mathbb{K}^\T$};
		\draw [decorate,decoration={brace,mirror,amplitude=4.5}] (Z.south west) -- (X.south east) node [black,midway,yshift=-0.5cm] {$\p3(\ker \p1)$};
		\draw [decorate,decoration={brace,amplitude=4.5}] (X.north west) -- (P4kerP2.north east) node [black,midway,yshift=0.5cm] {$\im \p4$};
\end{tikzpicture}
\end{center}
We define $\mathcal{X}_2=\p3(\ker \p1)\cap \im \p4$. $\mathcal{X}_2\cap \p4(\ker \p2)=\{0\}$ according to equation~\eqref{P3kerP1P4kerP2}. Then, we define $\mathcal{X}_3$ as a complement of $\p4(\ker \p2)\oplus \mathcal{X}_2$ in $\im \p4$ and $\mathcal{X}_1$ as a complement of $\mathcal{X}_2$ in $\p3(\ker \p1)$. $\mathcal{X}_4$ is defined as a complement of $\mathcal{X}_1\oplus\mathcal{X}_2\oplus\mathcal{X}_3\oplus \p4(\ker \p2)$.

Finally, we build $L$ through the associated mapping, itself defined using a direct sum of linear mappings defined on the following subspaces of $\mathbb{K}^\T$:
\begin{itemize}
\item We use Lemma~\ref{build} to construct a bijective linear mapping $f$ from $\mathcal{T}=\mathcal{X}_2\oplus \mathcal{X}_3$ over $\mathcal{Y}_1$. By definition, for all $v\in \mathcal{T}$, $f$ verifies $f(v)\in \p2\p4^{-1}(\{v\})$. Furthermore, as $f$ is bijective, its restriction on $\mathcal{X}_2$ is itself bijective over $f(\mathcal{X}_2)$.
\item We complete this bijective linear mapping with $g$, a bijective linear mapping between $\mathcal{X}_1$ and a complement $\mathcal{Y}_2$ of $f(\mathcal{X}_2)$ in $\mathcal{Y}$. This way, the restriction of $f\oplus g$ on $\mathcal{X}_1\oplus\mathcal{X}_2=\p3(\ker \p1)$ is bijective over $\mathcal{Y}$.
\item We consider the mapping $h$ that maps $\p4(\ker \p2)\oplus \mathcal{X}_4$ to $\{0\}$.
\end{itemize}
\begin{center}
\begin{tikzpicture}
		\node (X1) at (0,0) {$\mathcal{X}_1$};
		\node (X2) at (1.6,0){$\mathcal{X}_2\oplus\mathcal{X}_3$};
		\node (P4kerP2) at (4,0) {$\p4(\ker \p2)\oplus\mathcal{X}_4$};
		\node [below = of X1](Y1) {$\mathcal{Y}_2$};
		\node [below = of X2](Y2) {$\mathcal{Y}_1$};
		\node [below = of P4kerP2](nullS) {$\{0\}$};
		\draw [->] (X1)  edge node[right] {$g$}(Y1);
		\draw [->] (X2)  edge node[right] {$f$}(Y2);
		\draw [->] (P4kerP2)  edge node[right] {$h$}(nullS);

\end{tikzpicture}
\end{center}

The matrix associated with the linear mapping $f\oplus g\oplus h$ satisfies all the conditions of Lemma~\ref{cond1}, and is therefore an optimal solution.

This method is summarized in Algorithm~\ref{alg:L1}, which allows us to construct a solution for Theorem~\ref{master}. Its key part is the construction of a basis of $\mathcal F$, which uses a generalized pseudo-inverse $\p4^\dagger$ (resp. $\p2^\dagger$) of $\p4$ (resp. $\p2$), i.e., a matrix verifying $\p4\p4^\dagger\p4=\p4$ (resp. $\p2\p2^\dagger\p2=\p2$).  This algorithm is a main contribution of this article.

\begin{algorithm}\small
\caption{Constructing $L$ (Theorem~\ref{master}), case $\R2\leq \T+\K-\R1 - \R4$}
\label{alg:L1}
\begin{algorithmic}
	\REQUIRE $\T$,$\K$ and $P\in GL_{\T+\K}(\mathbb{K})$ such that $\R2\leq \T+\K-\R1 - \R4$
	\ENSURE $L$
	\STATE $Y_1\leftarrow$ Alg.~\ref{alg:dc} with $A=\p2\cdot \olker \p4$, $B=\p1\olcap \p2$, $C = \p2$
	\STATE $Y\leftarrow \begin{pmatrix}Y_1 & I_\T\olominus \begin{pmatrix}Y_1 & \p1\end{pmatrix}\end{pmatrix}$
	\STATE $X_2\leftarrow (\p3\cdot \olker \p1)\olcap \p4$
	\STATE $X_3\leftarrow \p4\olominus \begin{pmatrix}(\p4\cdot\ker\p2)& X_2\end{pmatrix})$ 
	\STATE $X_1\leftarrow (\p3\cdot\olker\p1) \olominus X_2$
	\STATE $X_4\leftarrow I_\T\olominus\begin{pmatrix}\p4 & \p3\cdot\olker\p1\end{pmatrix}$
	\STATE $F\leftarrow \begin{pmatrix}\olker \p4 & \p4^\dagger\cdot \begin{pmatrix}X_2 & X_3\end{pmatrix}\end{pmatrix}\olcap \begin{pmatrix}\olker \p2 & \p2^\dagger\cdot Y_1\end{pmatrix}$
	\STATE $Y_2\leftarrow Y\olominus (\p2\cdot(\begin{pmatrix}\olker \p4 &\p4^\dagger\cdot X_2\end{pmatrix}\olcap F))$
	\STATE $L_R\leftarrow \begin{pmatrix}\p4\cdot F & X_1 & \p4\cdot \olker\p2 & X_4\end{pmatrix}$ 
	\STATE $L_L\leftarrow \begin{pmatrix}\p2\cdot F & Y_2 & Z\end{pmatrix}$, where $Z$ is a zero filled matrix such that $L_L$ has the same number of columns as $L_R$
	\RETURN $L_L\cdot L_R^{-1}$
\end{algorithmic}
\end{algorithm}

Inspection of this algorithm shows that its arithmetic cost is $O((\T+\K)^3)$.

\subsection{Example}
We illustrate our algorithm with a concrete example. Motivated by our main application (Section~\ref{sec:SLP}) we choose as base field $\mathbb{K}=\mathbb{F}_2$. For $\T=4$ and $\K=3$, we consider the matrix 
$$
P=\begin{pmatrix}\p4 & \p3\\\p2 & \p1\\\end{pmatrix}=\left(\begin{array}{cccc|ccc}
1 & 1 &   &   & 1 &   & 1\\
  &   & 1 &   &   & 1 &  \\
1 & 1 &   & 1 & 1 & 1 & 1\\
  &   &   &   &   & 1 & 1\\
\hline
1 &   & 1 & 1 &   &   &  \\
1 & 1 &   &   & 1 & 1 & 1\\
  & 1 &   &   &   &   &  \\
\end{array}\right ).
$$

We observe $\R2=3\leq\T+\K-\R1-\R4=4+3-1-3$. Therefore, we can use Algorithm~\ref{alg:L1} to compute a suitable $L$.

The first step is to compute $Y_1$. We have:$$\p2\cdot \olker \p4=
\begin{pmatrix}1\\0 \\1\\\end{pmatrix}
\text{ and }\p1\olcap\p2=
\begin{pmatrix}0 \\1\\0 \\\end{pmatrix}
.$$ Using Algorithm~\ref{alg:dc}, we get $$Y_1=
\begin{pmatrix}1&1\\1&0\\1&0\\\end{pmatrix}
.$$ Then, we complete it to form a complement of $\olim\p1$: $$Y=
\begin{pmatrix}1&0\\0&1\\0&1\\\end{pmatrix}.
$$
The next step computes the different domains:
$$X_2=
\begin{pmatrix}1\\1\\ 0\\ 0\\\end{pmatrix}
\text{, }X_3=
\begin{pmatrix}1 \\ 0\\0\\0 \\\end{pmatrix}
\text{, }X_1=
\begin{pmatrix} 0\\0 \\ 0\\1\\\end{pmatrix}
\text{ and }X_4=()
.$$
To compute $F$, we need pseudo-inverses of $\p4$ and $\p3$:
$$\p4^\dagger=\begin{pmatrix}0 &0&0&0\\1&0&0&0 \\0 &1&0&0\\1&0&1&0\\\end{pmatrix}\text{ and }\p2^\dagger=\begin{pmatrix}0 &1&1\\0&0&1 \\ 0&0&0\\1&1&1\\\end{pmatrix}.
$$
We then obtain: 
$$
F=\begin{pmatrix} 0&0\\1&0 \\ 0&1\\1&0\end{pmatrix}.
$$
Then, we compute $Y_2$:
$$
\mathcal{Y}_2=
\begin{pmatrix}1 \\ 0\\0\\\end{pmatrix}.
$$
Now we can compute $L$. With
$$
L_R=\begin{pmatrix}
1 &  0 &  0 & 0 \\
0 & 1  & 0  & 1 \\
0 & 0 & 0 & 1\\
0 & 0 & 1 & 0\\
\end{pmatrix},
L_L=\begin{pmatrix}
1 &  1 & 1  & 0 \\
1 &  0 & 0  &  0\\
1 & 0 & 0 & 0\\
\end{pmatrix}
$$
we get
$$
L=L_L\cdot L_R^{-1}=
\begin{pmatrix}
1 &  1 &  1 & 1 \\
1 & 0  &  0 & 0 \\
1 & 0 & 0 & 0\\
\end{pmatrix}.
$$

The final decomposition is now obtained using \eqref{expdeco}:
\begin{align*}P=&\left(\begin{array}{cccc|ccc}
1 &   &   &   &  &  & \\
  & 1 &   &   &  &  & \\
  &   & 1 &   &  &  & \\
  &   &   & 1 &  &  & \\
\hline
1 &1   & 1  & 1  & 1 &   &  \\
1 &   &   &   &   & 1 &  \\
1 &  &  &  &   &   & 1\\
\end{array}\right )\cdot\left(\begin{array}{cccc|ccc}
 & 1  &  &  & 1 &   & 1\\
  &   & 1 &   &   & 1 &  \\
 &  1 &  & 1  & 1 & 1 & 1\\
 1 &  &  &  &   & 1 & 1\\
\hline
 &  &  &  &  & 1  & 1\\
 &  &  &  &   & 1 &  \\
 &  &  &  &  1 &  & 1\\
\end{array}\right )\cdot\\&\left(\begin{array}{cccc|ccc}
1 &   &   &   &  &  & \\
  & 1 &   &   &  &  & \\
  &   & 1 &   &  &  & \\
  &   &   & 1 &  &  & \\
\hline
  &   &   &   & 1 &   &  \\
  &   &   &   &   & 1 &  \\
 1 &  &  &  &   &   & 1\\
\end{array}\right ).\end{align*}

This decomposition satisfies $\rank L=2$ and $\rank R=1$, thus matching the bounds of Theorem~\ref{bounds}.

If we consider the application presented in Section~\ref{sec:SLP}, this decomposition provides a way to implement in hardware the permutation associated with $P$ on $128$ elements, arriving in chunks of $8$ during $16$ cycles. 
This yields an implementation consisting of a permutation network of $4$ \ttt-switches, followed by a block of $8$ RAM banks, followed by another permutation network with $8$ \ttt-switches.

\section{Proof of Theorem~\ref{master}, case $\R2\geq \T+\K-\R1 - \R4$}
\label{p2bigger}
In this case, the third inequality in Theorem~\ref{bounds} is restrictive (Figure~\ref{boundGraph} right). Using again Lemma~\ref{rewriting}, we have to build a matrix $L$ satisfying:
$$\begin{cases}
\p1-L\p3 \text{ is non-singular}\\
\rank L = \K-\R1\\
\rank (\p2-L\p4) = \R1+\R2-\K
\end{cases}$$

As in the previous section, we will first provide a set of sufficient conditions for $L$ and then build it.
\subsection{Sufficient conditions}
\label{Scond2}
The set of conditions that we will derive in this subsection will be slightly more complex than in the previous section, as we cannot reach the intrinsic bound of $\rank (\p2-L\p4)$. Particularly, we cannot use Lemma~\ref{P2SP4min1} directly.
\begin{lem}
\label{cond2}
If $\mathcal{W}$ is such that $\mathcal{W}\oplus \im \p1=\mathbb{K}^\K$ and $\mathcal{T}$ is such that:
$$\left\lbrace \begin{array}{l} 
\mathcal{T}\cap \p4(\ker \p2)=\{0\}\\
\mathcal{T} \leq \im \p4\\
\dim \mathcal{T} =\K-\R1,\\
\end{array}\right.$$
then if $L$ satisfies
$$\left\lbrace \begin{array}{l} 
\im L = \mathcal{W}\\
L\cdot \p3(\ker \p1)=\mathcal{W}\\
L\cdot v\in \p2\p4^{-1}(\{v\}), \forall v \in \mathcal{T}\\
L\cdot \p4(\ker \p2) = \{0\}\\
\end{array}\right.$$
then $L$ is a  solution\footnote{If we replace the last condition with $L\cdot \p4(\ker \p2)\leq \p2(\ker \p4)$, this set of conditions is actually equivalent to having an optimal solution $L$ that satisfies $\rank L=\K-\R1$.} that verifies $\rank L= \K-\R1$ and $\rank (\p2-L\p4) = \R1+\R2-\K$.
\end{lem}
\begin{proof}
Let $L$ be a matrix that satisfies the conditions above. Using Lemma~\ref{isSol} as before, we get that $\rank L=\K-\R1$ and $\p1-L\p3$ invertible.

Now, with the definition of $\mathcal{T}$, we can define a dimension $\R1+\R2+\R4-\T-\K$ space $\mathcal{T'}$ such that $\im \p4=\p4(\ker \p2)\oplus \mathcal{T} \oplus \mathcal{T'}$. Then, we define a matrix $L'$ such that:
$$\left\lbrace \begin{array}{l} 
L'\cdot v\in Lv-\p2\p4^{-1}(\{v\}) \text{, for all } v \in \mathcal{T'}\\
L'\cdot v=0 \text{, for all $v$ in $\p4(\ker \p2)\oplus \mathcal{T}$}\\
L'\cdot v=0 \text{, for all $v$ in a complement of $\im \p4$},\\
\end{array}\right.$$
$L'$ is therefore a rank $\R1+\R2+\R4-\T-\K$ matrix such that $\text{for all }v\in \im \p4, (L-L')v \in \p2\p4^{-1}(\{v\})$. We apply Lemma~$\ref{P2SP4min1}$ on $L-L'$ and get:
$$\begin{array}{rcl}
\rank(\p2-L\p4) & = & \rank (\p2-(L-L')\p4 - L'\p4)\\
& \leq & \rank (\p2-(L-L')\p4) + \rank(L'\p4)\\
& \leq & \dim \p2(\ker \p4) + \rank L'\\
& \leq & \R1+\R2-\K,\\
\end{array}$$
as desired.
\end{proof}

\subsection{Building $L$}\label{Final2}

We will build a matrix $L$ that matches the conditions listed in Lemma~\ref{cond2}. As before, we consider the image $\mathcal{Y}$ of $L$ first. We will design it such that it is a complement of $\im \p1$, and that is contained in a complement $\mathcal{Y'}$ of $\p2(\ker \p4)$ in $\im \p2$. Using $\R2\geq \T+\K-\R1-\R4$ and Lemma~\ref{dcomplement}, we can get a space $\mathcal{Y}$ such that:
$$\left \{\begin{array}{l}
\mathcal{Y}\oplus (\im \p1 \cap \im \p2)=\im \p2\\
\mathcal{Y}\cap \p2(\ker \p4)=\{0\}
\end{array}\right .$$ 
This space satisfies $\mathcal{Y}\oplus \im \p1=\im \p2 + \im \p1=\mathbb{K}^\K$, and can be completed to a complement $\mathcal{Y'}$ of $\p2(\ker \p4)$ in $\im \p2$. Note that we will use $\mathcal{Y'}$ only to define $f$; the image of $L$ will be $\mathcal{Y}$.

Now, as before, we build $L$ through the associated mapping, itself defined using a direct sum of linear mappings defined on the same subspaces of $\mathbb{K}^\T$ as in Section~\ref{Final1}.
\begin{itemize}
\item We use Lemma~\ref{build} to construct a first bijective linear mapping $f'$ between $\mathcal{T'}=\mathcal{X}_2\oplus \mathcal{X}_3$ and $\mathcal{Y'}$. As $f'$ is bijective, we can define $\mathcal{T}=f'^{-1}(\mathcal{Y})$ and $f=\restriction{f'}{\mathcal{T}}$. Thus, $\mathcal{T}$ satisfies the properties in Lemma~\ref{cond2} and $L$ the condition for all $v\in \mathcal{T}, Lv\in \p2\p4^{-1}(\{v\})$.
\item Then, we consider a complement $\mathcal {X}_1'$ of $\mathcal{T}\cap \mathcal{X}_2$ in $\p3(\ker \p1)$, a complement $\mathcal{Y}_2'$ of $f(\mathcal{T}\cap \mathcal{X}_2)$ in $\mathcal{Y}$ and a bijective linear mapping $g$ between $\mathcal {X}_1'$ and $\mathcal {Y}_2'$. This way, the restriction of $f\oplus g$ on $\p3(\ker \p1)$ is bijective over $\mathcal{Y}$.
\end{itemize}
The rest of the algorithm in similar to the previous case:
\begin{itemize}
\item We consider the mapping $h$ that maps $\p4(\ker \p2)$ to $\{0\}$.
\item To complete the definition of $L$, we take a mapping $h'$ between a complement $\mathcal{X}_4'$ of $\mathcal{X}_1'\oplus\mathcal{T}\oplus \p4(\ker \p2)$ and $\{0\}$.
\end{itemize}

\begin{center}
\begin{tikzpicture}
		\node (X1) at (0,0) {$\mathcal{X}_1'$};
		\node (X2) at (1.6,0){$\mathcal{T}=f'^{-1}(\mathcal{Y})$};
		\node (P4kerP2) at (4,0) {$\p4(\ker \p2)$};
		\node (X4) at (6,0){$\mathcal{X}_4'$};
		\node [below = of X1](Y1) {$\mathcal{Y}_2'$};
		\node [below = of X2](Y2) {$\mathcal{Y}$};
		\node [below = of P4kerP2](nullS) {$\{0\}$};
		\draw [->] (X1)  edge node[right] {$g$}(Y1);
		\draw [->] (X2)  edge node[right] {$f$}(Y2);
		\draw [->] (P4kerP2)  edge node[right] {$h$}(nullS);
		\draw [->] (X4)  edge node[right] {$h'$}(nullS);
\end{tikzpicture}
\end{center}

The matrix associated with the mapping $f\oplus g\oplus h\oplus h'$ satisfies all the conditions of Lemma~\ref{cond2}, and is therefore an optimal solution.

Algorithm~\ref{alg:L2} summarizes this method, and allows to construct a solution for Theorem~\ref{master}, in the case where $\R2> \T+\K-\R1 - \R4$. This algorithm is a main contribution of this article.

\begin{algorithm}
\caption{Constructing $L$ (Theorem~\ref{master}), case $\R2> \T+\K-\R1 - \R4$}\small
\label{alg:L2}
\begin{algorithmic}
	\REQUIRE $\T$,$\K$, $P\in GL_{\T+\K}(\mathbb{K})$ such that $\R2> \T+\K-\R1 - \R4$
	\ENSURE $L$
	\STATE $Y\leftarrow$ Alg.~\ref{alg:dc} with $A=\p1\olcap\p2$, $B=\p2\cdot\olker\p4$ and $C=\p2$
	\STATE $X_2\leftarrow (\p3\cdot\olker\p1)\olcap \p4$
	\STATE $X_3\leftarrow \p4\olominus \begin{pmatrix}\p4\cdot\olker\p2 & X_2\end{pmatrix}$
	\STATE $F\leftarrow\begin{pmatrix}\olker \p4 &\p4^\dagger\cdot\begin{pmatrix}X_2 & X_3\end{pmatrix}\end{pmatrix}\olcap \begin{pmatrix}\olker \p2 &\p2^\dagger\cdot Y\end{pmatrix}$
  \STATE $X_1'\leftarrow (\p3\cdot\ker\p1) \olominus ((\p4\cdot F)\olcap X_2)$
	\STATE $X_4\leftarrow I_\T\olominus\begin{pmatrix}X'_1&\p4\cdot F&\p4\cdot\olker\p2\end{pmatrix}$
	\STATE $Y_2'\leftarrow Y\olominus( \p2\cdot(F\olcap\begin{pmatrix}\p4^\dagger\cdot X_2 & \olker \p4\end{pmatrix}))$
	
  \STATE $L_R\leftarrow \begin{pmatrix}\p4\cdot F & X'_1 & \p4\cdot \olker\p2 & X_4\end{pmatrix}$ 
	\STATE $L_L\leftarrow \begin{pmatrix}\p2\cdot F & Y'_2 & Z\end{pmatrix}$, where $Z$ is a zero filled matrix such that $L_L$ has the same number of columns as $L_R$
	\RETURN $L_L\cdot L_R^{-1}$

\end{algorithmic}
\end{algorithm}

As in the previous case, this algorithm has an arithmetic cost cubic in $\T+\K$.

\subsection{Example}

\label{sec:ex2}
We now consider, for $\mathbb{K}=\mathbb{F}_2$, $\T=4$ and $\K=3$, the matrix $$P=\begin{pmatrix}\p4 & \p3\\\p2 & \p1\\\end{pmatrix}=\left(\begin{array}{cccc|ccc}
  & 1 & 1 & 1 & 1 &   &  \\
1 &   &   & 1 &   & 1 & 1\\
  & 1 & 1 & 1 &   & 1 & 1\\
1 & 1 &   & 1 &   & 1 & 1\\
\hline
1 &   &   & 1 &   & 1 &  \\
  &   &   & 1 &   & 1 &  \\
1 &   & 1 & 1 & 1 & 1 &  \\
\end{array}\right ).$$

We observe $\R2=3>\T+\K-\R1-\R4=4+3-2-3$. Therefore, we use Algorithm~\ref{alg:L2} to compute a suitable $L$.

The first step is to compute $Y$. We have:$$\p2\cdot \olker \p4=
\begin{pmatrix}0\\1 \\1\\\end{pmatrix}
\text{ and }\p1\olcap\p2=
\begin{pmatrix}0 \\1\\0 \\\end{pmatrix}
.$$ Using Algorithm~\ref{alg:dc}, we get $$Y_1=
\begin{pmatrix}1&1\\1&0\\1&0\\\end{pmatrix}
.$$ Then, we complete it to form a complement of $\im\p1$: $$Y=
\begin{pmatrix}1&0\\1&0\\0&1\\\end{pmatrix}.$$

The next step computes the different domains:
$$
X_2=()
\text{ and }X_3=
\begin{pmatrix}1&0 \\ 0&1\\1&0\\0&0 \\\end{pmatrix}.
$$

To compute $F$, we need pseudo-inverses of $\p4$ and $\p3$:
$$\p4^\dagger=\begin{pmatrix}0&0&0&0\\0&1&0&1 \\0 &0&1&1\\0&1&0&0\\\end{pmatrix}\text{ and }\p2^\dagger=\begin{pmatrix}1 &1&0\\0&0&0 \\ 1&0&1\\0&1&0\\\end{pmatrix}$$
and get 
$$F=\begin{pmatrix} 1\\1 \\ 0\\0\end{pmatrix}.$$

Next we compute the remaining subspaces that depend on $F$: 
$$
X_1'=
\begin{pmatrix}0 \\ 1\\1\\1\end{pmatrix}\text{, } X_4=
\begin{pmatrix}1 \\ 0\\0\\0\end{pmatrix}\text{, and } Y_2'=
\begin{pmatrix}1 \\ 0\\1\end{pmatrix}
.$$

Noe we can compute $L$. With
$$L_R=\begin{pmatrix}
1 &  0 &  1 & 1 \\
1 & 1  & 0  & 0 \\
1 & 1 & 1 & 0\\
0 & 1 & 1 & 0\\
\end{pmatrix},
L_L=\begin{pmatrix}
1 &  1 & 0  & 0 \\
0 &  0 & 0  &  0\\
1 & 1 & 0 & 0\\
\end{pmatrix}
$$
we get
$$
L=L_L\cdot L_R^{-1}=
\begin{pmatrix}
0 &  1 &  0 & 0 \\
0 & 0  &  0 & 0 \\
0 & 1 & 0 & 0\\
\end{pmatrix}
.$$

The final decomposition is obtained using \eqref{expdeco}:
\begin{align*}P=&\left(\begin{array}{cccc|ccc}
1 &   &   &   &  &  & \\
  & 1 &   &   &  &  & \\
  &   & 1 &   &  &  & \\
  &   &   & 1 &  &  & \\
\hline
 &1   &   &   &  &   &  \\
 &   &   &   &   &  &  \\
 & 1 &  &  &   &   & \\
\end{array}\right )\cdot\left(\begin{array}{cccc|ccc}
 & 1  &  & 1     & 1 &   & \\
 1 &   &  &      &   & 1 & 1 \\
 &  1 & 1 &      &  & 1 & 1\\
 1 & 1 &  &      &   & 1 & 1\\
\hline
 &  &  &  &  &   & 1\\
 &  &  &  &   & 1 &  \\
 &  &  &  &  1 &  & 1\\
\end{array}\right )\cdot\\&\left(\begin{array}{cccc|ccc}
1 &   &   &   &  &  & \\
  & 1 &   &   &  &  & \\
  &   & 1 &   &  &  & \\
  &   &   & 1 &  &  & \\
\hline
  &   & 1  &   & 1 &   &  \\
  &   &   & 1  &   & 1 &  \\
  &  &  &  &   &   & 1\\
\end{array}\right ).\end{align*}

This decomposition satisfies $\rank L=1$ and $\rank R=2$, thus matching the bounds of Theorem~\ref{bounds}.

As before, if we consider the application of Section~\ref{sec:SLP}. 
The decomposition shows that we can implement in hardware the permutation associated with $P$ on $128$ elements, arriving in chunks of $8$ during $16$ cycles through a permutation network of $8$ \ttt-switches, followed by a block of $8$ RAM banks, followed by another permutation network with $4$ \ttt-switches.


\section{Rank exchange}\label{rkexc}

The solution built in Section~\ref{Final2} satisfies $\rank L=\K-\R1$ and $\rank \p2-L\p4=\R1+\R2-\K$. In this section, we will show that it is possible to construct a solution for all possible pairs $(\rank L, \rank \p2-L\p4)$ matching the bounds in Theorem~\ref{bounds}. First, we will construct a rank $1$ matrix $L'$ that will trade a rank of $L$ for a rank of $\p2-L\p4$ (i.e., $\rank (L+L')= 1+\rank L$, $\rank (\p2-(L+L')\p4)=\rank(\p2-L\p4)-1$ and $\p1-(L+L')\p3$ is non-singular) (see Figure~\ref{fig:rkexc}). This method can then be applied several times, until $\rank (\p2-L\p4)$ reaches its own bound, $\T-\R4$.

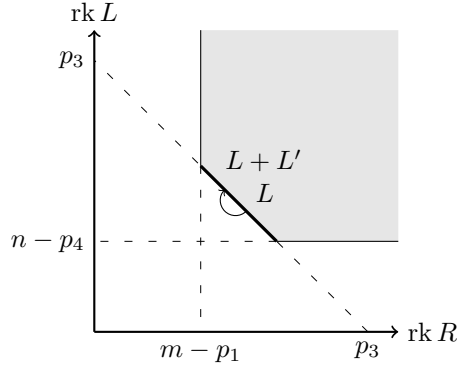
\begin{figure}
\centering
\begin{tikzpicture}[scale=0.4]
\draw[->,thick](0,0) -- (10,0);
\draw[->,thick](0,0) -- (0,10);
\draw (10,0) node [right] {$\rank R$};
\draw (0,10) node [above] {$\rank L$};
\path [fill=gray!20] (3.5,5.5)--(6,3)--(10,3)--(10,10)--(3.5,10)--(3.5,5.5);
\draw (6,3) -- (10,3);
\draw (3.5,5.5) -- (3.5,10);
\draw[loosely dashed,very thin] (3.5,5.5)--(3.5,0);
\draw[loosely dashed,very thin] (6,3)--(0,3);
\draw (3.5,0) node [below] {$\T-\R4$};
\draw (0,3) node [left]  {$\K-\R1$};
\draw[loosely dashed,very thin] (9,0)--(0,9);
\draw (9,0) node [below] {$\R2$};
\draw (0,9) node [left]  {$\R2$};
\draw[very thick] (3.5,5.5)--(6,3);
\draw (5,4) node [above right] {$L$};
\draw (4,5) node [above right] {$L+L'$};
\draw[->] (5,4) arc (-45:-225:0.5);
\end{tikzpicture}
\caption{$L'$ trades a rank of $L$ for a rank of $R$ on the associated decomposition.}
\label{fig:rkexc}
\end{figure}

We assume that $L$ satisfies the following conditions:
$$\left \{\begin{array}{l}
\p1-L\p3\text{ is non-singular}\\
\rank L+\rank(\p2-L\p4)=\R2\\
\rank(\p2-L\p4)>\T-\R4
\end{array}\right .$$

As in the previous sections, we first formulate sufficient conditions on $L'$, before building it.
\subsection{Sufficient conditions}
We now define $C=\p4-\p3(\p1-L\p3)^{-1}(\p2-L\p4)$. 
\begin{lem}
\label{rkexcond}
If $z\in\mathbb{K}^{\T}$ satisfies $z\notin\ker (\p2-L\p4)+\ker \p4$ and $L'$ satisfies:
$$\left \{\begin{array}{l}
\rank L'=1\\
L'\p4\ker(\p2-L\p4)=\{0\}\\
L'\p4z=(\p2-L\p4)z\\
L'Cz\neq0
\end{array}\right .$$ 
Then $L+L'$ is an optimal solution to our problem that satisfies $\rank (\p2-(L+L')\p4)=\rank (\p2-L\p4)-1$.
\end{lem}
\begin{proof} We first prove that $\p1-(L+L')\p3=(I-L'\p3(\p1-L\p3)^{-1})(\p1-L\p3)$ is non singular. Let $x\in \ker(I-L'\p3(\p1-L\p3)^{-1})$. $x$ satisfies:
$$x-L'\p3(\p1-L\p3)^{-1}x=0.$$
Therefore, $x\in \im L'$. As $\rank L'=1$, $\exists \lambda \in \mathbb{K}$, $x=\lambda(\p2-L\p4)z=\lambda L'\p4z$. It comes that $$\lambda L'\p4z-\lambda L'\p3(\p1-L\p3)^{-1}(\p2-L\p4)z=0.$$ Finally, $\lambda L' C z=0$, which implies, as $L' C z\neq0$, $\lambda=0$, as desired.
\item We now prove that $\rank (\p2-(L+L')\p4)=\rank (\p2-L\p4)-1$. We have already $\ker (\p2-(L+L')\p4)\leq \ker (\p2-L\p4)$ as $L'\p4\ker(\p2-L\p4)=\{0\}$. We also have $(\p2-(L'+L)\p4)z=0$. As $z\notin\ker (\p2-L\p4)+\ker \p4$, $\ker (\p2-(L+L')\p4)\geq \ker (\p2-L\p4)\oplus\langle z\rangle$. Therefore, $$\dim \ker (\p2-(L+L')\p4)\geq 1+\dim \ker (\p2-L\p4),$$ as desired.
\end{proof}

\subsection{Building $L'$}
\begin{lem}
$\ker (\p2-L\p4)\cap \ker \p4=\{0\}$
\end{lem}
\begin{proof}
This is a consequence of \eqref{decor}: the block column $\begin{pmatrix}\p4 \\ \p2-L\p4\end{pmatrix}$ has full rank.
\end{proof}

Thus, we have $\dim (\ker (\p2-L\p4)\oplus \ker \p4)=2\T-\R4-\rank(\p2-L\p4)<\T$. Decomposition~\eqref{expdeco} shows that $C$ is non-singular, and we have $\dim C^{-1}\p4\ker (\p2-L\p4)=\dim \p4\ker (\p2-L\p4)=\dim \ker (\p2-L\p4)=\T-\rank (\p2-L\p4)$. Using Lemma~\ref{dcomplement}, we can build a space $\mathcal{Z}$ such that:
$$\left \{\begin{array}{l}
 \mathcal{Z} \oplus \ker (\p2-L\p4)\oplus \ker \p4 =\mathbb{K}^{\T}\\
\mathcal{Z}\cap C^{-1}\p4\ker (\p2-L\p4) = \{0\}
\end{array}\right .$$ 

We can now pick a nonzero element $z\in\mathcal{Z}$ and build a corresponding $L'$:
\begin{itemize}
\item If $Cz\in\p4\ker (\p2-L\p4)\oplus \langle \p4z \rangle$: We take a complement $\mathcal{A}$ of $\p4\ker (\p2-L\p4)\oplus \langle \p4z \rangle$ and build $L'$ such that:
$$\left \{\begin{array}{l}
L'\p4z=(\p2-L\p4)z\\
L'(\p4\ker(\p2-L\p4)\oplus \mathcal{A})=\{0\}\\
\end{array}\right .$$ 
We have $L'Cz\neq0$. In fact, $Cz$ can be uniquely decomposed in the form $k+\lambda P_4z$, where $k\in \p4\ker (\p2-L\p4)$ and $\lambda\in\mathbb{K}$. As $z\notin C^{-1}\p4\ker (\p2-L\p4)$, $\lambda \neq 0$. Then, $L'Cz=L'k+L'\lambda P_4z=0 + \lambda (\p2-L\p4)z\neq 0$.

\item If $Cz\notin\p4\ker (\p2-L\p4)\oplus \langle \p4z \rangle$: The vector $a=Cz-\p4z$ is outside of $\p4\ker (\p2-L\p4)\oplus \langle \p4z \rangle$. Therefore, it is possible to build a complement $\mathcal{A}$ of $\p4\ker (\p2-L\p4)\oplus \langle \p4z \rangle$ that contains $a$. Then, we build $L'$ as before:
$$\left \{\begin{array}{l}
L'\p4z=(\p2-L\p4)z\\
L'(\p4\ker(\p2-L\p4)\oplus \mathcal{A})=\{0\}\\
\end{array}\right .$$ 
As in the previous case, we have $L'Cz=L'a+L'\p4z=0+(\p2-L\p4)z\neq0$.
\end{itemize}
In both cases, the matrix $L'$ we built satisfies the conditions of Lemma~\ref{rkexcond}. Therefore, $L+L'$ is the desired solution.

Algorithm~\ref{alg:Lp} summarizes this method, and allows to build a new optimal solution from a pre-existing one, with a different trade-off. This algorithm is a main contribution of this article.

\begin{algorithm}\small
\caption{Exchanging ranks between $L$ and $R$ (Theorem~\ref{master2})}
\label{alg:Lp}
\begin{algorithmic}
	\REQUIRE $P$ and a solution $L$ such that $\rank(\p2-L\p4)>\T-\R4$
	\ENSURE A new optimal solution $L$ with a rank incremented by $1$
	\STATE $K\leftarrow\olker(\p2-L\p4)$
	\STATE $C\leftarrow \p4-\p3(\p1-L\p3)^{-1}(\p2-L\p4)$
	\STATE $Z\leftarrow$ Alg. \ref{alg:dc} with $A= \begin{pmatrix}K &\olker\p4\end{pmatrix}, B=C^{-1}\p4 K$ and $C=I_m$
	\STATE $z\leftarrow$ first column of $Z$
	\IF {$Cz\in \begin{pmatrix}\p4\cdot K& \p4 z\end{pmatrix}$}
	\STATE $A\leftarrow I_m\olominus \p4\cdot\begin{pmatrix} K & z\end{pmatrix}$
	\ELSE
	\STATE $a\leftarrow (C-\p4)z$
	\STATE $A\leftarrow \begin{pmatrix} I_m\olominus\begin{pmatrix}\p4 K & \p4 z & a\end{pmatrix} & a\end{pmatrix}$
	\ENDIF
	\STATE $L'_R\leftarrow \begin{pmatrix}\p4z& \p4K& A\end{pmatrix}$
		\STATE $L'_L\leftarrow \begin{pmatrix}(\p2-L\p4)z & F\end{pmatrix}$, where $F$ is a zero filled matrix such that $L'_L$ has the same number of columns as $L'_R$
	\RETURN $L+L'_L\cdot L_R'^{-1}$
\end{algorithmic}
\end{algorithm}

\subsection{Example}
To illustrate Algorithm~\ref{alg:Lp}, we continue the example of Section~\ref{sec:ex2}, and the matrix $L$ that we found. We have:$$K=
\begin{pmatrix}1&0\\0&1 \\0&0 \\0&0\\\end{pmatrix}
\text{ and } C=
\begin{pmatrix}
  & 1 &   & 1\\
 1 &   &  &  \\
 & 1 &  1 &  \\
1  & 1 &  &  \\
\end{pmatrix}
.$$ Using Algorithm~\ref{alg:dc}, we get $$Z=z=
\begin{pmatrix} 0\\ 0\\ 1\\0\end{pmatrix}
.$$ 
As $Cz=
\begin{pmatrix}0\\0 \\1 \\0 \end{pmatrix}
$ is not included in $\begin{pmatrix}\p4K & \p4z\end{pmatrix}=
\begin{pmatrix}1&0&0\\0&1&0 \\1&0&0\\0&0&1 \\\end{pmatrix}
$, we compute $A$ as a complement of $\begin{pmatrix}\p4K & \p4z\end{pmatrix}$ that contains $a=
\begin{pmatrix} 1\\0\\0\\0\\\end{pmatrix}
$:$$\mathcal{A}=
\left\langle\begin{pmatrix}1 \\0\\0\\0\\\end{pmatrix}\right\rangle
.$$

Now, we compute $L'$, using:$$L'_R=
\begin{pmatrix}1&0&1&1\\0&1 &0&0\\1&0&1&0 \\0&1&1&0\\\end{pmatrix}
\text{, } L'_L=
\begin{pmatrix}
 0 & 0 & 0  & 0\\
 0 & 0  & 0 & 0 \\
1 & 0 &  0 & 0 \\
\end{pmatrix}
\text{, and } L'=
\begin{pmatrix}
 0 & 0 & 0  & 0\\
 0 & 0  & 0 & 0 \\
0 & 1 &  1 & 1 \\
\end{pmatrix}
.$$

Finally, we get the new decomposition, using, as usual, Equation~\eqref{expdeco}:\begin{align*}P=&\left(\begin{array}{cccc|ccc}
1 &   &   &   &  &  & \\
  & 1 &   &   &  &  & \\
  &   & 1 &   &  &  & \\
  &   &   & 1 &  &  & \\
\hline
 & 1 &  &   & 1 &   &  \\
  &  &  &  &   & 1 &  \\
 &  & 1 &1   &   &   & 1\\
\end{array}\right )\cdot\left(\begin{array}{cccc|ccc}
  & 1 & 1  & 1 & 1 &   &  \\
1 &   &   &  &   & 1 & 1\\
  & 1 & 1 &  &   & 1 & 1\\
1 & 1 &   &  &   & 1 & 1\\
\hline
 &  &  &  &  &  & 1 \\
 &  &  &  &   &  1 & \\
 &  &  &  & 1  & 1 &  \\
\end{array}\right )\cdot\\&\left(\begin{array}{cccc|ccc}
1 &   &   &   &  &  & \\
  & 1 &   &   &  &  & \\
  &   & 1 &   &  &  & \\
  &   &   & 1 &  &  & \\
\hline
  &   &  &   & 1 &   &  \\
  &   &  & 1  &   & 1 &  \\
  &   &  &   &   &   & 1\\
\end{array}\right ).\end{align*}

As expected, the left off-diagonal rank has increased by one, while the right one has decreased by one. The two different decompositions that we now have for $P$ cover all the possible tradeoffs that minimize the off-diagonal ranks.


\section{Conclusion}

In this paper, we introduced a novel block matrix decomposition that generalizes the classical block-LU factorization. A \ttt-blocked invertible matrix is decomposed into a product of three matrices: lower block unitriangular, upper block triangular, and lower block unitriangular matrix, such that the sum of the off-diagonal ranks are minimal. We provided an algorithm that computes an optimal solution with an asympotic number of operations cubic in the matrix size. We note that we implemented the algorithm for finite fields, for rational numbers, for Gaussian rational numbers and for exact real arithmetic for validation. For a floating point implementation, numerical issues may arise.

The origin of the considered decomposition, as we explained, is in the design of optimal circuits for a certain class of streaming permutations that are very relevant in practice. However, we believe that because of its simple and natural structure, the matrix decomposition is also of pure mathematical interest. Specifically, it would be interesting to investigate if the proposed decomposition is a special case of a more general problem that involves, for example, finer block structures.

\bibliographystyle{plain}
\bibliography{bib}
\end{document}